\documentclass[12pt]{amsart}
\usepackage{amsmath,mathtools}
\usepackage{amssymb}
\usepackage{srcltx}
\usepackage{xcolor}
\usepackage{tikz}
\usepackage{tikz-cd}

\tikzset{
  symbol/.style={
    draw=none,
    every to/.append style={
      edge node={node [sloped, allow upside down, auto=false]{$#1$}}}
  }
}
\usepackage{mathabx}

\newcommand{\res}{\mathord{\upharpoonright}}

\newcommand{\Fraisse}{Fra\"iss\'e}

\newtheorem{theorem}{Theorem}
\newtheorem{lemma}{Lemma}

\newtheorem{corollary}[theorem]{Corollary}
\newtheorem{proposition}[theorem]{Proposition}
\newtheorem{claim}{Claim}
\newtheorem{fact}{Fact}

\theoremstyle{remark}
\newtheorem{rem}[theorem]{\bf Remark}
\newtheorem{question}{Question}
\newcommand{\C}{\mathcal{C}}
\newcommand{\K}{\mathcal{K}}

\newtheorem{example}[theorem]{Example}

\title[Universality vs Genericity and $C_4$-free graphs]{Universality vs Genericity and $C_4$-free graphs}

\author[A. Panagiotopoulos]{Aristotelis Panagiotopoulos}
\address{Institut f\"ur Mathematische Logik und Grundlagenforschung, Westfalische Wilhelms-Universit\"at M\"unster, Einsteinstr. 62, 48149 M\"unster, Germany 
}
\email{aristotelis.panagiotopoulos@gmail.com}

\author[K. Tent]{Katrin Tent}
\address{Institut f\"ur Mathematische Logik und Grundlagenforschung, Westfalische Wilhelms-Universit\"at M\"unster, Einsteinstr. 62, 48149 M\"unster, Germany}
\email{tent@wwu.de}

\thanks{
Funded by the Deutsche Forschungsgemeinschaft (DFG, German Research Foundation) under Germany's Excellence Strategy EXC 2044--390685587, Mathematics M\"unster: Dynamics--Geometry--Structure
and by CRC 1442 Geometry: Deformations and Rigidity.}



\begin{document}
\maketitle

\begin{abstract} We show that the existence of a universal structure implies the existence of a generic structure for any approximable class $\mathcal{C}$ of countable structures.  We also show  that the converse is not true. As a consequence, we provide several new examples  of weak Fra\"iss\'e classes of finite graphs. Finally, we show that the class of all countable $C_4$-free graphs does not contain a generic structure, strengthening a result of A. Hajnal  and J. Pach.

\end{abstract}

\section{Introduction}

An {\bf approximating class} is any family $\mathcal{K}$ of finitely generated $\mathcal{L}$-structures which is countable up to isomorphism and which has the  {\bf Hereditary Property} (HP) and the {\bf Extension Property} (EP):
\begin{itemize}
\item (HP) if $B\in \mathcal{K}$,  $A$  embeds in $B$, and $A$ is finitely generated, then $A\in \mathcal{K}$;
\item (EP)  every $A\in\mathcal{K}$ admits a non-surjective embedding to some $B\in\mathcal{K}$.
\end{itemize} 

Let $\mathcal{K}$ be an approximating class. A class $\mathcal{C}$ of countable  $\mathcal{L}$-structures is {\bf approximable} by $\K$ if $\C$ consists  precisely of all  $\mathcal{L}$-structures $N$ on universe $\mathbb{N}=\{0,1,\ldots\}$  that  satisfy the following properties:
\begin{itemize}
\item every finitely generated substructure $A$ of $N$ is in  $\mathcal{K}$;
\item $N$ is not finitely generated.
\end{itemize}
In this case, we write $\mathcal{K}=\mathrm{Age}(\mathcal{C})$ and $\mathcal{C}=\lim(\mathcal{K})$.  We say that $\mathcal{C}$ is {\bf approximable} if  $\mathcal{C}=\lim(\mathcal{K})$ for some approximating class $\mathcal{K}$.
Standard examples of approximable  classes are classes consisting of all countably infinite graphs  that omit all graphs which are included in some fixed   countable collection $\mathcal{F}$ of ``forbidden" finite graphs.

In what follows, we will always assume that $\mathcal{C}$ is approximable and hence  each structure in $\C$ has universe $\mathbb{N}$.
We say that $\mathcal{C}$  has a  {\bf universal element}  $U\in \mathcal{C}$, if every $N\in\mathcal{C}$ embeds in $U$. 
We say that  $\mathcal{C}$ has a {\bf generic element} $G\in \mathcal{C}$, if in the Baire space  $\mathcal{C}$   (see Section \ref{S:E}) the following set is  comeager:
\[\{N\in \mathcal{C} \; \colon  N \text{ is isomorphic to } G \}.\]

Several  approximable classes $\mathcal{C}$ admit a structure $M\in\mathcal{C}$ that is both universal and generic. For example, the  Rado graph is both a universal and a generic element of the class of all countable graphs; see \cite{Rado,TZ}.
More generally, if $\mathcal{K}=\mathrm{Age}(\mathcal{C})$ is a  \Fraisse{} class (see Section \ref{S:E}), then  $\mathcal{C}$ contains a canonical structure  $M\in \mathcal{C}$---known as  the   \Fraisse{} limit of $\mathcal{K}$---which is both universal and generic. However, it is possible for an approximable class $\mathcal{C}$ to  admit a universal $U\in\mathcal{C}$ that is not generic and a generic $G\in\mathcal{C}$ that is not universal; see Example \ref{E:1}. It is therefore natural to ask whether there is an example of an approximable class which admits a universal element but does not admit a generic element and similarly whether there is    an approximable class which admits a generic  element but does not admit a universal element. We settle this with the following theorem.

\begin{theorem}\label{T:1}
Let $\C$ be an approximable class of  $\mathcal{L}$-structures. If
 $\C$  admits a universal element, then $\mathcal{C}$ admits a generic element. The converse is not true.
\end{theorem}


As a consequence of Theorem \ref{T:1}, we get several new examples of approximable classes of countable graphs which  admit a generic element. Notice that,  by the Lachlan-Woodrow classification theorem \cite{Lachlan}, the \Fraisse{} construction---in its original form---provides  very few examples of approximable classes of graphs which admit a generic element. While we now know that a certain weak version of the \Fraisse{} construction is essentially the only way to produce generic elements~\cite{KKgames,K, KT, DiLiberti}, very few ``weak \Fraisse{} classes" of graphs have been explicitly constructed \cite{WAP}.  On the other hand,
there is a vast literature  on approximable classes  $\mathcal{C}$ of graphs which admit a universal element $U\in\mathcal{C}$;  e.g., \cite{Ko,KMP, KP, DHV,CT,CSS}. By Theorem \ref{T:1}, these classes also admit a generic element $G\in\mathcal{C}$.  We collect some  examples in Section \ref{S:E}.

Finally, there are several approximable classes $\mathcal{C}$ for which it has been established that they do not admit a universal element. A natural question stemming from Theorem \ref{T:1}  is whether  results about non-existence of universal elements can be  strengthened to non-existence of generic elements. Here we strengthen  a result of  Hajnal and Pach  on the non-existence of a universal countable   $C_4$-free graph  \cite{HP}. Recall that  a graph $A$ is {\bf $C_4$-free} if there exists no injective homomorphism from $C_4$ to $A$, or equivalently, if $A$ does not contain a simple $4$-cycle.  We show the following:

\begin{theorem}\label{T:2}
The class of all countable $C_4$-free graphs admits no generic element.
\end{theorem}
By Theorem  \ref{T:1},  establishing  that an approximable class $\mathcal{C}$ does not admit a generic element  is a priori more difficult than establishing that $\mathcal{C}$ does not admit a universal element. The proof of Theorem \ref{T:2} demonstrates this in practice since, unlike with the arguments from \cite{HP}, its proof heavily relies on structure theorems for strongly regular graphs and $C_4$-free graphs of diameter $2$ from \cite{BB,DF}.

\subsection*{Acknowledgments} 
We would like to thank A. Kruckman and W. Kubi\'s{} for their valuable feedback on an earlier draft of this paper. We would also like to thank the anonymous referees whose  comments significantly improved the exposition of the paper. 
After the completion of this paper we were informed by  W.Kubi\'s that Theorem \ref{T:1} has independently been established in  \cite[Section 6]{KKgames}; see Remark \ref{R:1} below.  
The argument we provide here avoids the game-theoretic formalism of  \cite{KKgames} and the counterexample we supply for the  second statement of Theorem \ref{T:1} is  different and interesting on its own right since it  relates to the theory of bowtie-free  graphs.

\section{Definitions and Examples}\label{S:E}

We will be using the   usual model theoretic  notions of    $\mathcal{L}$-structures, embeddings, etc.; see e.g., \cite{TZ}.  In particular, if $A$ is a graph then we will denote by $\mathrm{dom}(A)$ the underlying set of all of its vertices and we will write $(a,a')\in R^A$ or $A\models R(a,a')$ to declare that there is an edge between $a,a'\in \mathrm{dom}(A)$ in $A$. Since the examples we consider  deal  exclusively with symmetric graphs, we will often view edges $(a,a')\in R^A$ as unordered pairs $\{a,a'\}$. An {\bf embedding} $f\colon A \to B$ of graphs is  any injective function $f\colon \mathrm{dom}(A)\to \mathrm{dom}(B)$ with $(a,a')\in R^A \iff (f(a),f(a'))\in R^B$. An injective homomorphism of graphs, or a {\bf weak embedding}, is  any injective function $f\colon \mathrm{dom}(A)\to \mathrm{dom}(B)$ with $(a,a')\in R^A \implies (f(a),f(a'))\in R^B$.
 We warn the reader that, in the graph theoretic terminology from \cite{Ko,KMP, KP, DHV,CT,CSS}, the term embedding is used for what we call here  weak embedding.

Below we will often use the notation $N\leq M$ to indicate that the $\mathcal{L}$-structure $N$ is a substructure of the $\mathcal{L}$-structure $M$. That is, to indicate that $\mathrm{dom}(N)\subseteq \mathrm{dom}(M)$ and that the inclusion $\mathrm{id}\colon \mathrm{dom}(N)\to \mathrm{dom}(M)$  induces an embedding of $N$ into $M$.  
Let $\mathcal{K}$ be a collection of $\mathcal{L}$-structures. We say that $\mathcal{K}$ has the  {\bf Joint  Embedding Property} (JEP)
 if for every pair  $A,B\in \mathcal{K}$ there is $C\in \mathcal{K}$ so that  $A,B$ embed in $C$. Let  $f\colon A\to B$ and $g\colon A\to C$ be embeddings with $A,B,C\in\mathcal{K}$. If there are embeddings $f'\colon  B\to D$ and $g'\colon C\to D$, with $D\in\mathcal{K}$, so that $f'\circ f = g' \circ g$ then we say that {\bf $f$ and $g$ amalgamate over $A$} (in $\mathcal{K}$). Whenever $f,g$ are unambiguously understood, we will  just say  that {\bf $B$ and $C$ amalgamate over $A$} and we will call $D$ an {\bf amalgam} of $B$ and $C$ over $A$. We say that $\mathcal{K}$ has {\bf  Amalgamation Property  (AP)} if for every pair of embeddings  $f\colon A\to B$ and $g\colon A\to C$, with $A,B,C\in\mathcal{K}$, we have that $f,g$  amalgamate over $A$. We say that  $\mathcal{K}$ has the {\bf Cofinal Amalgamation Property (CAP)} if for every $A\in\mathcal{K}$ there is an embedding $i\colon A\to\widehat{A}$ so that for every two embeddings $f\colon\widehat{A}\to B$ and $g\colon\widehat{ A}\to C$ we have that 
$f$ and $g$ amalgamate over $\widehat{A}$. We say that  $\mathcal{K}$ has the {\bf Weak Amalgamation Property (WAP)} if for all $A\in\mathcal{K}$ there is an embedding $i\colon A\to\widehat{A}$ so that for every two embeddings $f\colon\widehat{A}\to B$ and $g\colon\widehat{ A}\to C$, the maps 
$f \circ i$ and $g \circ i$ amalgamate over $A$.
Notice that:
\[\mathcal{K} \text{ has  AP}  \implies \mathcal{K} \text{ has  CAP}  \implies \mathcal{K} \text{ has  WAP}  \]

Assume that $\mathcal{K}$ is an approximating class, as defined in the  introduction.
Let $\mathcal{C}=\mathrm{lim}(\mathcal{K})$ be the associated approximable class and notice that, by EP, $\mathcal{C}\neq \emptyset$.  As usual, we view  $\mathcal{C}$  as a topological space  with basic open sets  of the form:
\begin{align*}
\mathcal{O}_A&=\{N\in  \mathcal{C} \colon 
\text{ the restriction of } N \text{ to } \mathrm{dom}(A) \text{ is equal to } A \}\\
&=\{N\in  \mathcal{C} \colon  A\leq N \},
\end{align*}
where $A$ ranges over all elements of $\mathcal{K}$ with  $\mathrm{dom}(A)\subseteq \mathbb{N}$; see e.g., \cite{KT}. 
 While $\mathcal{C}$ fails in general to be a Polish space (see  \cite[Proposition 2.3]{KT}), it turns out that  $\mathcal{C}$ is always a Baire space. That is, the intersection of countably many comeager subsets of $\mathcal{C}$ forms a dense subset of $\mathcal{C}$; see \cite[Definition 8.2 and Section 8.A]{Kechris}. 
 
\begin{proposition}
If $\mathcal{C}$ is an approximable class, then $\mathcal{C}$  is a Baire space.
\end{proposition}
 \begin{proof}
 By  \cite[Theorem 8.4]{Kechris} is suffices to show that $\mathcal{C}$ admits a complete metric compatible with the topology on $\mathcal{C}$. For any   $N,M\in \mathcal{C}$ let $\delta(N,M)\subseteq \mathbb{N}$ be the set of all $k\in\mathbb{N}$ so that the identity map $\mathrm{id}_k\colon \{0,1,\ldots,k-1\}\to \{0,1,\ldots,k-1\}$ extends to an isomorphism between the substructures $\langle  \{0,1,\ldots,k-1\} \rangle_M$  of  $M$ and $\langle  \{0,1,\ldots,k-1\} \rangle_N$ of  $N$ generated by $\{0,\ldots, k-1\}$. We set
\[d_{\mathcal{C}}(M,N):=\mathrm{inf}\{1/2^k\colon k\in \delta(N,M)\}.\]
Here we adopt the convention that the empty structure is an $\mathcal{L}$-structure and hence  $0\in\delta(N,M)$ for all $M,N\in\mathcal{C}$.
It is easy to see that  $d_{\mathcal{C}}$ is a metric on $\mathcal{C}$ that is compatible with the topology on $\mathcal{C}$.  To see that   $d_{\mathcal{C}}$ is complete let $(M_n)_{n\in\mathbb{N}}$ be a $d_{\mathcal{C}}$-Cauchy sequence. By passing to a subsequence we may assume that for all $n\in \mathbb{N}$ and every $m\geq n$, the substructure of $M_n$ generated by $\{0,\ldots, n-1\}$ is equal to the substructure of $M_m$ generated by $\{0,\ldots, n-1\}$. Call this common substructure $A_n$ and let $M$ be the $\mathcal{L}$-structure that is the union of the increasing sequence
\[A_0\leq  A_1 \leq A_2\leq \cdots \leq A_n\leq \cdots \]
Then $M\in\mathcal{C}$. Indeed, it is clear that $\mathrm{dom}(M)=\mathbb{N}$. Moreover, if $A$ is some finitely generated structure with $A\leq M$ then there is some $n\in\mathbb{N}$ so that $\mathrm{dom}(A_n)$ contains the generators of $A$ and hence  $A\leq A_n$.  By HP we have that $A\in \mathcal{K}$. Moreover, since $M_n$ is not finitely generated we have that $\mathrm{dom}(A)\neq \mathbb{N}$. Hence $M$ is not finitely generated either.  It is finally clear that $M_n\to M$.
 \end{proof}
 
Since $\C$ is a non-empty Baire space, the $\sigma$-filter  of all comeager subsets of $\mathcal{C}$ is non-trivial, i.e., the empty set is not comeager in $\C$; see \cite[Section 8.B]{Kechris}. As a consequence, the topology of $\mathcal{C}$ carries a useful  intrinsic notion of ``largeness" which allows one to talk about properties of the ``generic" structure in the same fashion that a probability measure on $\mathcal{C}$ would allow one to talk about properties of the``random" structure  of $\mathcal{C}$. The following definitions make this precise.
A property $\mathcal{P}$ of structures of $\mathcal{C}$ is a {\bf  generic property} if the set of all $N\in \mathcal{C}$  satisfying $\mathcal{P}$  forms a comeager set. We say that $\mathcal{C}$ {\bf has a  generic element} if there is $G\in\mathcal{C}$ so that
\[\{N\in \mathcal{C} \; \colon  N \text{ is isomorphic to } G \}\]
is a comeager subset of $\mathcal{C}$. In this case, we say that $G$ {\bf is a generic element of} $\mathcal{C}$. 

\begin{fact}\label{F:2}
If $\mathcal{C}$ is approximable then there exists at most one generic element of $\mathcal{C}$ up to isomorphism. 
\end{fact}
\begin{proof}
Let $G_1,G_2$ both be generic elements of $\mathcal{C}$. 
Since $\mathcal{C}$ is a Baire space, the set
 \[\mathcal{D}:=\{N \in \mathcal{C} \colon  N \text{ is isomorphic to } G_1 \} \cap \{N \in \mathcal{C} \colon  N \text{ is isomorphic to } G_2 \} \]
is dense in $\mathcal{C}$ and hence non-empty---recall that by EP we have that $\mathcal{C}\neq\emptyset$. Hence, $G_1$ and $G_2$ are isomorphic, since they are both isomorphic to a structure in $\mathcal{D}$. 
\end{proof}

Standard examples of  approximable classes $\mathcal{C}$ which admit a structure that is both generic and universal are constructed via the \Fraisse{} method.
We say that an approximating class $\mathcal{K}$ is a {\bf \Fraisse{} class} if $\mathcal{K}$ additionally  satisfies  JEP and  AP. In this case, the class $\mathcal{C}=\mathrm{lim}(\mathcal{K})$ contains a canonical structure $M\in\mathcal{C}$, known as the {\bf \Fraisse{} limit of $\mathcal{K}$}, which is both a universal and a generic element of $\mathcal{C}$; see \cite{TZ}.  For example, the class $\mathcal{K}_{\mathrm{graphs}}$ of all finite graphs forms  \Fraisse{} class whose  \Fraisse{} limit is the Rado graph $G_{R}$: the unique up to isomorphism countable graph  so that for every two finite disjoint  $F,F'\subseteq \mathrm{dom}(G_{R})$  there is a vertex $v\in\mathrm{dom}(G_{R})$ which is connected by an edge with every vertex in $F$ but to no vertex in $F'$.

The full strength of the amalgamation property for $\mathcal{K}$ is not necessary for the existence of a generic element of $\mathcal{C}=\mathrm{lim}(\mathcal{K})$.  Indeed, if $\mathcal{K}$ is a just {\bf weak \Fraisse{} class}, i.e.,  if $\mathcal{K}$ is an approximating class and satisfies  JEP  and WAP, then a variant of the \Fraisse{} construction shows that $\mathcal{C}$ still admits a generic element. In fact, as it turns out, 
weak \Fraisse{} classes  are in bijective correspondence to approximable classes  that  admit a generic element; see \cite[Theorem 4.2.2]{K} or \cite[Theorem 2.5]{KT}.  
 This was originally proved for certain special classes of partial automorphisms by Ivanov~\cite{Ivanov} 
and independently by Kechris and Rosendal~\cite{KR}. It was later extended to more general classes in \cite{K} and independently by several other authors in~\cite{KKgames, KT, DiLiberti}. Here we follow \cite{KT}, whose exposition is closer in spirit to this paper\footnote{Notice that in \cite{KT} the term ``$\mathcal{K}$ is unbounded" is used for what we call here $``\mathcal{K}$ has EP".}. 

\begin{fact}\label{F:1}  \text{\cite[Theorem 2.5(4)-(5)]{KT}}
Let $\mathcal{K}$ be an approximating class and  $\mathcal{C}=\mathrm{lim}(\mathcal{K})$ be the associated approximable class. Then $\mathcal{C}$ admits a generic element if and only if $\mathcal{K}$ has JEP and WAP.
\end{fact}

The following example demonstrates that, in the absence of the full amalgamation property for $\mathcal{K}$, the associated approximable class $\mathcal{C}$ may admit a universal element $U$ which is not generic, as well as a generic element $G$ which is not universal.

\begin{example}\label{E:1}
Let $\mathcal{K}_{\mathrm{linear}}$ be the class of all finite graphs which are disjoint unions of linear graphs.  By a linear graph we mean any graph which is isomorphic to the graph on domain $\{0,\ldots,n-1\}$ with $(k R \ell)\iff (|k-\ell|=1)$, for some $n\in\mathbb{N}$. It is easy to see that  $\mathcal{K}_{\mathrm{linear}}$  has  CAP and hence, by Fact \ref{F:1},  $\mathcal{C}_{\mathrm{linear}} := \lim (\mathcal{K}_{\mathrm{linear}})$ admits a generic element $G$. In a personal communication, A. Kruckman pointed out that while $G$ is generic, it  is not universal for  $\mathcal{C}_{\mathrm{linear}}$. Indeed, $G$  is isomorphic to the bi-infinite linear graph i.e., the
 graph whose domain is $\mathbb{Z}$ and for $k,\ell\in\mathbb{Z}$ we have $(k R \ell)\iff (|k-\ell|=1)$;  see 
 \cite[Example 4.2.7]{K}. In particular,  $G$ does not embed any  extension of $G$ from $\mathcal{C}_{\mathrm{linear}}$, such as the graph $G\sqcup G$ that is the disjoint union of two copies of $G$. That being said,  $\mathcal{C}_{\mathrm{linear}}$ admits a universal element. Namely,  the graph $\bigsqcup_{n\in \aleph_0} G$ that is the disjoint union of $\aleph_0$-many copies of $G$. 
 
\end{example}

As an application of  Theorem \ref{T:1},  we close this section by demonstrating that many well-studied classes $\mathcal{K}$ of finite graphs are actually weak \Fraisse{} classes. Notice that, by the Lachlan-Woodrow classification theorem \cite{Lachlan}, the only 
 \Fraisse{} classes of finite graphs are: the class of  all graphs, classes consisting of  disjoint unions of complete graphs; the class of all $K_n$-free graphs;  complements of the last two classes.

Let $\mathcal{F}$ be a collection of finite graphs. We say that a graph {\bf $N$ omits graphs from $\mathcal{F}$}, if there exists no injective homomorphism $f\colon A\to N$ with $A\in \mathcal{F}$. Let $\mathcal{K}(\mathcal{F})$ and $\mathcal{C}(\mathcal{F})$ be all finite and all countable, respectively, graphs which omit graphs from $\mathcal{F}$. Below,  $C_n$ and $K_n$ denote  the cycle and the complete graph of size $n$, respectively.

\begin{example}\label{E:2}
The  class $\mathcal{C}(\mathcal{F})$ admits a universal element if $\mathcal{F}$ consists of: 
\begin{enumerate}\setlength\itemsep{5pt}\vspace{2pt}
\item the singleton $\{\bowtie\}$, where   $\bowtie$ is  the bowtie graph; see \cite{Ko}.
\item  the infinite collection  $\{C_n,C_{n+1},\ldots\}$, for some $n>0$; see \cite{KMP}.
\item the finite collection  $\{C_3,C_5,\ldots,C_{2n+1}\}$, for some  $n\geq 1$; see \cite{KMP} and  \cite{Ko}.
\item  the singleton $\{P_n\}$, where $P_n$ is the  length $n$ path, for some $n>0$; see  \cite{KMP}.
\item the singleton $\{N\}$,  where $N$ is a near path, i.e.,  a finite tree consisting of a path with at most one additional edge adjoined; see \cite{CT}.
\item the collection $\mathrm{top}(K_n)$ for some fixed $n\leq 4$,  where $\mathrm{top}(K_n)$ consists of all finite graphs which are topologically equivalent to $K_n$, i.e., all finite graphs which are  simplicial subdivisions of $K_n$; see \cite{DHV} and \cite{KP}.
\item any finite collection  of finite   connected graphs  that is  homomorphism-closed, i.e., if $f\colon B\to A$ is a surjective graph homomorphism with with $B\in \mathcal{F}$ then there is a $C\in \mathcal{F}$ which weakly embeds in $A$; see \cite{CSS}.
\end{enumerate} 
\end{example}

\begin{corollary}\label{Cor:1}
$\mathcal{K}(\mathcal{F})$ is a weak \Fraisse{} class, for all
$\mathcal{F}$ from Example \ref{E:2}.
\end{corollary}

The fact that $\mathcal{K}(\{\bowtie\})$ has WAP was already known. In fact, in 
\cite{HN} it is shown that  $\mathcal{K}(\{\bowtie\})$ has the CAP; see also \cite{S}.  It turns out that the associated generic bowtie-free countable graph has several curious properties, both from a Ramsey-theoretic as well as from  a model-theoretic standpoint; see \cite{HN}.

We should point out that, even when it comes to  general classes of $\mathcal{L}$-structures, we have very few (and rather artificial) examples of weak \Fraisse{} classes $\mathcal{K}$ which do not already satisfy the cofinal amalgamation property; see \cite{WAP}.  It would be interesting if any of the examples (2)-(7) above
could provide more natural examples.

\begin{question} Does  $\mathcal{K}(\mathcal{F})$ fail CAP for some
 $\mathcal{F}$ from (2)-(7) in Example \ref{E:2}?
\end{question}


\section{Universality vs Genericity}

We first show that the  existence of a universal $U\in\mathcal{C}$  implies the existence of a generic element $G\in\mathcal{C}$. By Fact \ref{F:1} it suffices to show the following:

\begin{theorem}\label{T:11}
Let $\mathcal{C}$ be an approximable class of $\mathcal{L}$-structures and let $\mathcal{K}=\mathrm{Age}(\mathcal{C})$. If $\mathcal{C}$  admits a universal element then $\mathcal{K}$ has JEP and WAP.
\end{theorem}
\begin{proof}
Let $A,B\in \mathcal{K}$. Since $\mathcal{K}$ has EP one can easily build  $M,N\in \mathcal{C}$ with $A\leq M$ and $B\leq N$. Indeed using EP we can inductively define strictly increasing sequences:
\begin{align*}
A=A_0\lneq A_1\lneq A_2\lneq\cdots\\
B=B_0\lneq B_1\lneq B_2\lneq\cdots
\end{align*}
of structures in $\mathcal{K}$. Let $M$ and $N$ be the union of the sequences $(A_n)_{n\in\mathbb{N}}$ and $(B_n)_{n\in\mathbb{N}}$ respectively. But then $M$ and $N$ are not finitely generated since if, say $F\subseteq \mathrm{dom}(M)$ is finite, then  $F\subseteq \mathrm{dom}(A_n)$ for some $n\in\mathbb{N}$ and hence $\langle F\rangle_M\leq A_n$. 
By renaming the domains of $M,N$ if necessary, we may assume that  $M,N\in\mathcal{C}$. Since $M,N$ embed in the universal element $U$ we may assume  that $A,B\leq U$. Consider the substructure $C:=\langle\mathrm{dom}(A)\cup\mathrm{dom}(B)\rangle_U$ of $U$ that is generated by   $\mathrm{dom}(A)\cup\mathrm{dom}(B)$. Notice that $C\in \mathcal{K}$ and $A,B\leq C$. Hence, $\mathcal{K}$ has JEP.

We now recall some standard pieces of notation which will be used for showing WAP. As usual we identify $n\in \mathbb{N}$ with $\{0,\ldots,n-1\}$. Let $2^{\mathbb{N}}$ and
$2^{<\mathbb{N}}:=\bigcup_{n\in\mathbb{N}}2^n$ be the sets of all infinite and finite sequences, respectively, with values in $\{0,1\}$. For every $\alpha=(\alpha_0,\alpha_1,\ldots)\in 2^{\mathbb{N}}$ and every $n\in\mathbb{N}$, we denote the  sequence $(\alpha_0,\ldots,\alpha_{n-1})\in 2^{<\mathbb{N}}$ by $\alpha|n$. The concatenation $(s_0,\ldots,s_{n-1},s'_0,\ldots,s'_{m-1})$ of $s=(s_0,\ldots,s_{n-1})$ and $s'=(s'_0,\ldots,s'_{m-1}) \in 2^{<\mathbb{N}}$ is denoted by $s^{\frown}s'$. We will write $s\precneq t$ if $t=s^{\frown}s'$ for some non-empty $s'\in 2^{<\mathbb{N}}$.

Now assume towards contradiction that  $\mathcal{K}$ fails WAP and  fix some $A\in\mathcal{K}$ so that for every embedding $i\colon A\to \widehat{A}$ with $\widehat{A}\in\mathcal{K}$ there are further extensions $B,C$ of $\widehat{A}$ which do not amalgamate over $A$. 
\begin{claim}\label{C:new1}
There is a family $(A_s\colon s\in 2^{<\mathbb{N}})$ with $A_s\in\mathcal{K}$  and  $A_{\emptyset}=A$ so that:
\begin{enumerate}
\item for all $s,t\in 2^{<\mathbb{N}}$ with  $s\precneq t$ we have $A_s\lneq A_t$;
\item for all $s\in 2^{<\mathbb{N}}$ we have that $A_{s^{\frown}(0)}$ and $A_{s^{\frown}(1)}$ do not amalgamate over $A_{\emptyset}$.
\end{enumerate}
\end{claim}
\begin{proof}[Proof of Claim]
Set  $A_{\emptyset}:=A$. Assume by induction that  $A_{s}\in\mathcal{K}$ has been defined for  $s\in 2^{<\mathbb{N}}$, and that $A_{\emptyset}\leq A_s$. Setting $\widehat{A}:=A_s$ and letting $i\colon A\to \widehat{A}$ be the inclusion  $A_{\emptyset}\leq A_s$, we get  $B,C\in\mathcal{K}$ with $\widehat{A}\leq B,C$ so that $B$ and $C$ do not amalgamate over $A$. Set $A_{s^{\frown}(0)}:=B$ and $A_{s^{\frown}(1)}:=C$. This concludes the construction.  It is clear that if $s\precneq t$ then we have $A_s\leq A_t$. In fact we have $A_s\lneq A_t$, since otherwise $A_{s^{\frown}(0)}:=B$ and $A_{s^{\frown}(1)}$ would amalgamate over $A_s$ and hence over $A_{\emptyset}$ as well.
\end{proof}

For every $\alpha\in 2^{\mathbb{N}}$ let $N_{\alpha}$ be the countable structure that is the union of 
\[A_{\alpha|0}\lneq A_{\alpha|1}\lneq \cdots \lneq  A_{\alpha|n}\lneq \cdots\]
Since $A=A_{\emptyset}=A_{\alpha|0}$ we have that $A\leq N_{\alpha}$ for all $\alpha\in 2^{\mathbb{N}}$. Moreover notice that,  since the above sequence is strictly increasing, $N_{\alpha}$ is not finitely generated and hence  $N_{\alpha}$ is isomorphic to some $M_{\alpha}\in \mathcal{C}$ for all $\alpha\in 2^{\mathbb{N}}$. Let $i_{\alpha}\colon A\to M_{\alpha}$ be the embedding that is induced by  postcomposing the inclusion $A\leq N_{\alpha}$ with the  isomorphism $N_{\alpha}\simeq M_{\alpha}$.

To derive a contradiction, let $U$ be the universal element of  $\mathcal{C}$  and  for every  $\alpha\in 2^{\mathbb{N}}$ fix some embedding $f_{\alpha}\colon M_{\alpha}\to U$. Since $U$ is countable and $A$ is finitely generated, there is an uncountable $J\subseteq 2^{\mathbb{N}}$ so that for all $\alpha,\beta\in J$ we have  that \[f_{\alpha} \circ i_{\alpha}=f_{\beta} \circ i_{\beta}.\]
Let $\alpha,\beta\in J$ with $\alpha \neq \beta$. But then, if $n\in \mathbb{N}$ is the largest number with $\alpha|n=\beta|n$, we have that $f_{\alpha} \circ i_{\alpha}=f_{\beta} \circ i_{\beta}$ is in contradiction with Claim \ref{C:new1}(2) for  $s:=\alpha|n$.
\end{proof}

\begin{rem}\label{R:1}
In \cite[Corollary 6.3]{KKgames}, A. Krawczyk and W. Kubi\'s prove the following strengthening of Theorem \ref{T:11}: ``if  there exists a family $\mathcal{U}\subseteq \mathcal{C}$ with $|\mathcal{U}|<2^{\aleph_0}$, so that for all $N\in\mathcal{C}$ there is $U\in\mathcal{U}$ with $N\leq U$, then $\mathcal{K}$ has WAP".
A straightforward adaptation of the above proof can be used to establish the aforementioned strengthening. Indeed, if  $\mathcal{U}$ is universal---as a family---for $\mathcal{C}$, with $|\mathcal{U}|<2^{\aleph_0}$,  then one can first choose some uncountable $I\subseteq 2^{\mathbb{N}}$ and some $U\in\mathcal{U}$ so that $M_{\alpha}$ embeds in $U$, for all $\alpha\in I$. The rest of the proof is verbatim, modulo requiring that $J\subseteq I$.
\end{rem}

Next we show that  the converse statement is not true. That is, there is some approximable class $\mathcal{C}$ which admits a generic element but no universal element.
Let $H$ be the \emph{windmill graph}, i.e., the graph  consisting of $3$ triangles all sharing exactly one common vertex $p$, as shown in the following figure:

 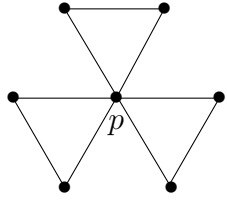
\begin{figure}[ht!]
\begin{tikzpicture}[x=0.52pt,y=0.52pt,yscale=-1,xscale=1]

\draw    (0,75)  node (1){$\bullet$} -- (150,75)  node (2){$\bullet$};
\draw    (37.50,10)  node (3){$\bullet$} -- (115,140) node (4){$\bullet$} ;
\draw    (75,75)  node (5){$\bullet$} -- (37.50,140)  node (6){$\bullet$} ;
\draw    (110,10)  node (7){$\bullet$} -- (75,75) ;
\draw    (37.50,10) -- (110,10) ;
\draw    (0,75) -- (37.50,140) ;
\draw    (150,75) -- (112.5,140) ;
\draw    (75,95)  node (5){$p$} ;
\end{tikzpicture}
\caption{The \emph{windmill graph}  $H$ together with its center $p$.}\label{F:Wheel}
 \end{figure}

 Let $\mathcal{K}(H)$ be the collection of all finite graphs which omit the graph $H$. That is, all graphs $A$ for which there is no injective homomorphism $f\colon H\to A$.
  In \cite{Ko} it is proved that $\mathcal{C}(H):=\lim (\mathcal{K}(H))$ does not admit a universal element.

\begin{theorem}
The class $\mathcal{C}(H)$ admits a generic element. 
\end{theorem} 
\begin{proof}
 By Fact \ref{F:1} it suffices to show that  $\mathcal{K}(H)$ has JEP and WAP. The fact that   $\mathcal{K}(H)$ has JEP is clear since for any two graphs $A,B\in \mathcal{K}(H)$, the disjoint union $A\sqcup B$ of $A$ and $B$ is also in $\mathcal{K}(H)$. So we are left to show that   $\mathcal{K}(H)$ has  WAP.
 
Let $A\in \mathcal{K}(H)$. We will define an extension $\widehat{A}\in\mathcal{K}(H)$ of $A$ so that any two further extensions $B,C\in\mathcal{K}(H)$ of $\widehat{A}$  amalgamate over $A$ to an element of $\mathcal{K}(H)$. 
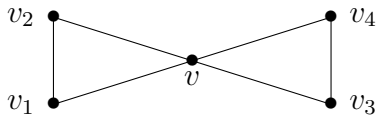
\begin{figure}[ht!]
\begin{tikzpicture}[x=0.35pt,y=0.22pt,yscale=-1,xscale=1]
\draw   (150,75) node (1){$\bullet$} -- (0,0) node (2){$\bullet$} -- (0,150) node (3){$\bullet$} -- cycle;
\draw   (150,75) --  (300,150) node (4){$\bullet$} -- (300,0) node (5){$\bullet$}-- cycle ;
\draw   (150,105) node (1'){$v$};
\draw   (-35,0) node (12){$v_2$};
\draw   (-35,150) node (12){$v_1$};
\draw   (335,0) node (12){$v_4$};
\draw   (335,150) node (12){$v_3$};
\end{tikzpicture}
\caption{ The bowtie graph $\bowtie$ together with a fixed labeling.}\label{F:bowtie}
 \end{figure}
 
Let $\bowtie$ be the bowtie graph, endowed with the labeling  from Figure \ref{F:bowtie}.

\begin{claim}\label{Claim1}
 There is graph $A^{\bowtie}\in \mathcal{K}(H)$ with $A\leq A^{\bowtie}$, so that for every vertex $a$ of  $A^{\bowtie}$ there is an injective homomorphism $j\colon\bowtie{}\to A^{\bowtie}$, with $j(v)=a$.
\end{claim}
\begin{proof}[Proof of Claim]
Let $X$ be the collection of all those vertices $a$ of $A$ for which there is no  injective homomorphism $j\colon\bowtie{}\to A$, with $j(v)=a$. In other words, $X$ is the collection of all $a\in\mathrm{dom}(A)$ for which there do not exist two  $3$-cycles $\{b_1,b_2,a\}$ and $\{c_1,c_2,a\}$  in $A$,   containing $a$, so that $\{b_1,b_2\}\cap \{c_1,c_2\}=\emptyset$. We partition $X=X_0\sqcup X_1$ into two pieces: the set $X_1$ which consists of those vertices $a$ of $A$ for which there is a  $3$-cycle $\{a,b_1,b_2\}$ in 
 $A$ containing $a$; the set $X_0$ which consists of those vertices $a$ of $A$ for which there is no $3$-cycle in $A$ containing $a$.
 
 One could try to extend $A$ by adding: for every $a\in X_1$ a  new triangle $\{p,q,r\}$ under the single identification $a=p$; and for every $a\in X_0$ two new triangles $\{p_1,q_1,r_1\}$ and $\{p_2,q_2,r_2\}$ under the single identification $a=p_1=p_2$. The resulting graph, call it $A'$, would clearly be an element of $\mathcal{K}(H)$ and for every $a\in  \mathrm{dom}(A)$ we would have some injective homomorphism $j\colon\bowtie{}\to A'$, with $j(v)=a$. However, the newly introduced vertices of the form $q,r$, or $q_1,r_1$, $q_2,r_2$, would not be the centers of a bowtie in $A'$. Hence $A'$ would fail to satisfy the property of the claim. 
 
 To remedy this we consider the graph $P$ from Figure \ref{F:P,p} together with its specified vertex $p$. It is clear that $P\in\mathcal{K}(H)$ since the degree of each vertex of $P$ is strictly less than $6$.  Notice moreover that for every vertex $q$ in $P$, with  $q\neq p$, there is an injective homomorphism $j\colon\bowtie{}\to P$, with $j(v)=q$. 
\begin{figure}[ht!]
\tikzset{every picture/.style={line width=0.75pt}} 

\begin{tikzpicture}[x=0.75pt,y=0.4pt,yscale=-1,xscale=1]
\draw   (125,260) node (1'){$p$};
\draw    (180,160) node (2){$\bullet$} --  (140,260) node (1){$\bullet$}; 
\draw    (180,360) node (3){$\bullet$} -- (140,260) ;
\draw    (180,160) -- (180,360) ;
\draw    (180,160) -- (300,100) node (4){$\bullet$} ;
\draw    (180,360) -- (300,420) node (5){$\bullet$} ;
\draw    (300,100) -- (280,180) node (6){$\bullet$} ;
\draw    (280,340)node (7){$\bullet$} -- (300,420) ;
\draw    (180,160) -- (280,180) ;
\draw    (280,340) -- (180,360)  ;
\draw    (320,240) node (8){$\bullet$} -- (320,280) node (9){$\bullet$} ;
\draw    (320,240) -- (260,260) node (10){$\bullet$};
\draw    (260,260) -- (320,280)   ;
\draw    (460,240) node (11){$\bullet$} -- (460,280) node (12){$\bullet$} ;
\draw    (460,240) -- (400,260) node (13){$\bullet$} ;
\draw    (400,260) -- (460,280) ;
\draw    (280,180) -- (320,240) ;
\draw    (280,180) -- (320,280) ;
\draw    (280,180) -- (260,260) ;
\draw    (320,240) -- (280,340) ;
\draw    (320,280) -- (280,340) ;
\draw    (260,260) -- (280,340) ;
\draw    (300,420) -- (460,280) ;
\draw    (400,260) -- (300,420) ;
\draw    (460,240) -- (300,100) ;
\draw    (300,100) -- (460,280) ;
\draw    (300,100) -- (400,260) ;
\draw  (300,420) --  (460,240);
\end{tikzpicture} 
 \caption{The graph $P$ together with its specified vertex $p$.}\label{F:P,p}
 \end{figure}
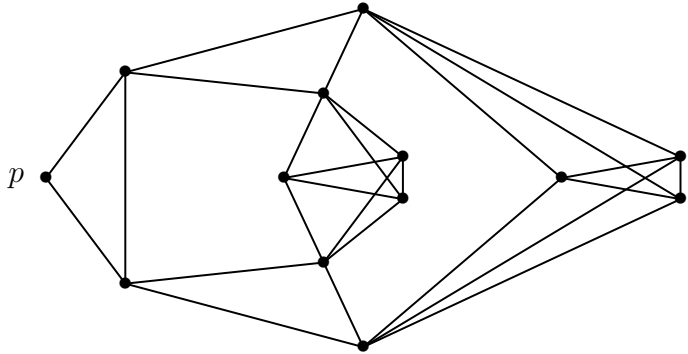

The graph $A^{\bowtie}$ is  attained by gluing several copies of $P$ to $A$ by identifying  the vertex corresponding to $p$ of each piece to some vertex $a\in X$. More precisely, for each $a\in X_1$ we introduce one new  copy of $P$ and we glue it to $A$ under the single identification $a=p$. For every $a\in X_0$ we introduce two new disjoint copies $P_1$ and $P_2$ of $P$ and we let $p_1\in\mathrm{dom}(P_1)$ and $p_2\in\mathrm{dom}(P_2)$ be the vertices of $P_1$ and $P_2$ corresponding to the vertex $p$ of $P$ from Figure \ref{F:P,p}. We then glue $P_1$ and $P_2$ on $A$ under the single  identification $p_1=p_2=a$.
It is easy to check that the resulting  extension $A^{\bowtie}$ of $A$ is in $\mathcal{K}(H)$ and that for every  vertex $a$ of $A^{\bowtie}$  there is an  injective homomorphism $j\colon\bowtie{}\to A^{\bowtie}$, with $j(v)=a$ 
 \end{proof}

\begin{claim}\label{Claim2}
There is some $\widehat{A}\in\mathcal{K}(H)$ with $A^{\bowtie}\leq \widehat{A}$ so that for every edge $\{a_1,a_2\}$ in $A^{\bowtie}$ the following holds: if there is a further extension $E\geq \widehat{A}$ of $\widehat{A}$ with $E\in\mathcal{K}(H)$ and a vertex $w\in\mathrm{dom} (E)\setminus \mathrm{dom}(A^{\bowtie})$ so that $\{w,a_1,a_2\}$ forms a $3$-cycle in $E$, then there is some $\hat{w}\in \mathrm{dom}(\widehat{A}) \setminus \mathrm{dom}(A^{\bowtie}) $ whose set of neighbors in $\widehat{A}$ is precisely $\{a_1,a_2\}$.
\end{claim}
 \begin{proof}[Proof of Claim]
We build $\widehat{A}$ by a simple induction. First let $\{\{a^k_1,a^k_2\}\colon  k\in \{1,\ldots,n\} \}$ be an enumeration of all edges of $A^{\bowtie}$ and set $A_0:= A^{\bowtie}$. Assume that $A_{k-1}$ has been defined for some $k\in \{1,\ldots,n-1\}$. We define $A_k$ as follows.  If there is no  extension $E\geq A_{k-1}$ of $A_{k-1}$ with $E\in\mathcal{K}(H)$ having a vertex $w\in\mathrm{dom} (E)\setminus \mathrm{dom}(A^{\bowtie})$ so that $\{w,a_1,a_2\}$ forms a $3$-cycle in $E$, then let $A_k:=A_{k-1}$. Otherwise, fix such $E$ and $w$, and  let $A_k$  be the graph attained by adding a new vertex $\widehat{w}^k$ to $A_{k-1}$ and connecting it only to $a^k_1$ and $a^k_2$. Notice that $A_{k}\in \mathcal{K}(H)$. Indeed, assume towards contradiction that there was some injective homomorphism  $j\colon H \to A_k$. Since $A_{k-1}\in\mathcal{K}(H)$ we have that $\widehat{w}^k\in\mathrm{range}(j)$. Let $j'\colon H \to E$  be defined by setting $j'(j^{-1}(\widehat{w}^k))=w$ and letting $j'=j$ on $\mathrm{dom}(H)\setminus \{j^{-1}(\widehat{w}^k)\}$. Notice that $j'$ is  an injective homomorhism, contradicting that  $E\in\mathcal{K}(H)$. To conclude the proof we set $\widehat{A}:=A_n$

 \end{proof}


 Let now $B,C\in \mathcal{K}(H)$ be any two extensions $\widehat{A}\leq B,C$  of $\widehat{A}$. We will show that $B,C$ amalgamate over $A$ to some $D\in\mathcal{K}(H)$.  Indeed, let $D$ be the  {\bf free amalgam} $D:=B\sqcup_{A^{\bowtie}} C$ of $B$ and $C$ over $A^{\bowtie}$. That is, $D$ is the graph attained by taking the  disjoint union  of the graphs $B$ and $C$, and identifying the copy of each vertex $a$ of $A^{\bowtie}$ which lies in $B$, with the associated copy of $a$ which lies in $C$. 
  We  view $B$ and $C$ as subgraphs of $D$ under the obvious identifications and  use the notation $\widehat{A}^B$ and $\widehat{A}^C$  for the copies of $\widehat{A}$ in $B\leq D$ and $C\leq D$ respectively. These graphs share $A^{\bowtie}$ as a common subgraph, but $\big(\mathrm{dom}(\widehat{A}^B)\setminus\mathrm{dom}(A^{\bowtie})\big)\cap \big(\mathrm{dom}(\widehat{A}^C)\setminus\mathrm{dom}(A^{\bowtie})\big)=\emptyset$. 
  
  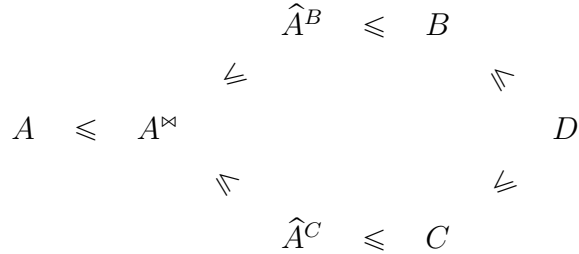
\begin{figure}[ht!]
 \begin{tikzcd}
\; & \; & \widehat{A}^B \arrow[r,symbol=\leq]  &  B \arrow[dr, symbol=\leq] & \\
 A \arrow[r,symbol=\leq] & A^{\bowtie}   \arrow[ur, symbol=\leq]   \arrow[dr, symbol=\leq]  &   &   & D\\ 
\; & \; & \widehat{A}^C \arrow[r,symbol=\leq]  &  C \arrow[ur, symbol=\leq]  & \\
\end{tikzcd}
\caption{The stratification of the graph $D$ into the pertinent subgraphs.}
\label{F:amalgam}
 \end{figure}


We are left to show that $D\in \mathcal{K}(H)$. Assume towards a contradiction that there is some injective homomorphism $j\colon H\to D$. Let $a:=j(p)$ be the image of the center $p$ of $H$ under $j$. By the structure of  the free amalgam $D:=B\sqcup_{A^{\bowtie}} C$,  it follows that $a$ lies in $A^{\bowtie}$: otherwise the copy of $H$ would already have to be entirely included  either in  $B$ or in $C$, contradicting that $B\in\mathcal{K}(H)$ or  $C\in\mathcal{K}(H)$, respectively. Since $a$ lies in $A^{\bowtie}$, by Claim \ref{Claim1} there is an injective homomorphism $i\colon \bowtie \to A^{\bowtie}$ with $i(v)=a$. Let   $a_1,a_2,a,a_3,a_4\in\mathrm{dom}(A^{\bowtie})$ be the images of the vertices $v_1,v_2,v,v_3,v_4$ under $i$.
 
\begin{claim}\label{Claim3}
If $\{a,d_1,d_2\}$ is a $3$-cycle in $D$, then  $\{a,d_1,d_2\}$ is entirely included either in $B$ or in $C$. Moreover, $\{d_1,d_2\}\cap  \{ a_1,a_2,a_3,a_4\}\neq \emptyset$.
\end{claim}
\begin{proof}[Proof of Claim]
The first statement follows  from the structure of the free amalgam $D:=B\sqcup_{A^{\bowtie}} C$. But then, since $A^{\bowtie}\leq B,C$,  if  $\{d_1,d_2\}\cap  \{ a_1,a_2,a_3,a_4\}= \emptyset$, the  $3$-cycles $\{a,d_1,d_2\}$, $\{a,a_1,a_2\}$, $\{a,a_3,a_4\}$ would form the three wings of a copy of $H$ in either $B$ or $C$. This would contradict either $B\in\mathcal{K}(H)$ or   $C\in\mathcal{K}(H)$.    
\end{proof}

 Let $\{b_1,b_2,a\}, \{c_1,c_2,a\}, \{c'_1,c'_2,a\}$ be the three cycles of length $3$ in $D$, which correspond to the three wings  of $j(H)$.  
By the first statement of Claim \ref{Claim3}  we may assume without loss of generality that $b_1,b_2\in\mathrm{dom}(B)$ and
 $c_1,c_2,c'_1,c'_2\in \mathrm{dom}(C)$.  By the second statement of Claim \ref{Claim3}  we may also assume that  $b_1=a_1$. 
 It follows that $b_2\in \mathrm{dom}(B)\setminus \mathrm{dom}(A^{\bowtie})$, since otherwise, the wing $\{b_1,b_2,a\}$ would  entirely lie in $ A^{\bowtie}\leq  C$,  contradicting that $C\in \mathcal{K}(H)$. 
  By Claim \ref{Claim2} we can find some  
vertex   $\hat{w}$ in  $\widehat{A}$ whose set of neighbors within $\widehat{A}$ are precisely $\{a_1,a\}$. Let $\hat{w}^C$ be the associated vertex of $C$ corresponding to $\hat{w}$ in the copy $\widehat{A}^C\leq C$ of  $\widehat{A}$ within $C$. Notice that since $\widehat{A}^C$ is an embedded copy of $\widehat{A}$ in $C$, the intersection of the set of neighbors of $\hat{w}^C$ in $C$ with the set $\{a,a_1,a_2,a_3,a_4\}\subseteq \mathrm{dom}(A^{\bowtie})$ is $\{a_1,a\}$.

Consider the map $j'\colon H\to C$ which is defined as $j$, with the only exception that $j'(j^{-1}(b_2))=\hat{w}^C$. Clearly $j'$ is a homomorphism mapping the three wings of $H$ to the cycles  $\{ a_1,\hat{w}^C,a\}, \{c_1,c_2,a\}, \{c'_1,c'_2,a\}$.  We will show that $j'$ is also injective, which will conclude the proof by contradicting that $C\in \mathcal{K}(H)$. 

To show that $j'$ is injective it suffices to show that  $\hat{w}^C\not\in \{c_1,c_2,c'_1,c'_2\}$.  By Claim \ref{Claim3} we have that 
$\{c_1,c_2\}\cap  \{ a_1,a_2,a_3,a_4\}\neq \emptyset$ and $\{c'_1,c'_2\}\cap  \{ a_1,a_2,a_3,a_4\}\neq \emptyset$. Moreover, since by injectivity of $j$ the six vertices $a_1=b_1,b_2,c_1,c_2,c'_1,c'_2$ are pairwise distinct, we  in fact  have  that  $\{c_1,c_2\}\cap  \{ a_2,a_3,a_4\}\neq \emptyset$ and $\{c'_1,c'_2\}\cap  \{a_2,a_3,a_4\}\neq \emptyset$.
But then  it follows that $\hat{w}^C\not\in \{c_1,c_2,c'_1,c'_2\}$. Indeed, if say $\hat{w}^C=c_1$, then since $\hat{w}^C\not \in \mathrm{dom}(A^{\bowtie})$ we would have $c_2\in \{ a_2,a_3,a_4\}$. But since $\hat{w}^C=c_1$ and $c_2$ are neighbors, this would contradict the aforementioned observation that, in $C$, the only neighbors  of $\hat{w}^C$ from the set $\{a,a_1,a_2,a_3,a_4\}$ are $a_1$ and $a$.

This concludes the proof of   $\hat{w}^C\not\in \{c_1,c_2,c'_1,c'_2\}$, yielding an injective homomorphism $j'\colon H\to C$, which contradicts that $C\in \mathcal{K}(H)$. Hence $D\in\mathcal{K}(H)$.

\end{proof}

We have established that if $\mathcal{C}=\mathrm{lim}(\mathcal{K})$  has a universal element, then  $\mathcal{K}$ satisfies WAP (and JEP) and that the converse is not true.  This concludes the proof of Theorem \ref{T:1}.
In the context of
 Theorem  \ref{T:1} and Fact \ref{F:1} it is very natural to ask:

\begin{question}
Is there an amalgamation property for $\mathcal{K}$, strictly lying between AP and WAP, which characterizes  when $\mathcal{C}=\mathrm{lim}(\mathcal{K})$  admits a universal element?\end{question}

\section{$C_4$-free graphs have no generic element}

In this section we prove Theorem \ref{T:2}. Let $\mathcal{K}(C_4)$ be the collection of all finite graphs which omit the graph $C_4$. That is, all graphs $A$ for which there is no injective homomorphism $f\colon C_4\to A$ from the $4$-cycle graph $C_4$ into $A$.
 By Fact \ref{F:1} it suffices to show that  the class $\mathcal{K}(C_4)$ does not have WAP.

We will need the following proposition. Let $v$ be a vertex of a graph $A$. We say that $v$ is {\bf dominating} if there is an edge between $v$ and any other vertex  of $A$. 

\begin{proposition}\label{prop:Blokhuis}
If $A\in \mathcal{K}(C_4)$ is of diameter $2$ without dominating vertex, then there is an edge in $A$ not contained in any $3$-cycle.

\end{proposition}
\begin{proof}

Let $A$ be a $C_4$-free graph of diameter $2$ without dominating vertex. Then
by Theorem~1 of~\cite{BB} we are in one of two cases:\begin{enumerate}
\item  the graph $A$ is strongly regular, i.e., there are integers $k, \lambda, \mu$ such that every vertex has degree $k$, every two adjacent vertices have $\lambda$ common neighbours and every two non-adjacent vertices have $\mu$ common neighbours.  Note that we have $\mu=1$.
\item the graph $A$ has two degrees, and in this case there are edges not
contained in a triangle (expressed in \cite{BB} as saying there are maximal
cliques having $2$ elements).
\end{enumerate}

In case (1) above,  we have $\lambda\neq 1$ (see e.g. \cite{DF} or  \cite{Ca}, Theorem 2.4), and so in
this case $A$ contains no 3-cycles.

Thus, any $C_4$-free graph of diameter 2 without dominating vertex has an
edge not contained in a triangle.
\end{proof}

Let $E$ be a finite graph and let $X\subseteq \mathrm{dom}(E)$. A vertex $v$ of $E$ is {\bf determined over $X$} if it belongs to the smallest set of vertices  $Z\subseteq\mathrm{dom}(E)$ so that $X\subseteq Z$, satisfying:  if $v\in \mathrm{dom}(E)$ and there are $x,y\in Z$ with $x\neq y$, $E\models (v R x)$,
 and $E\models (v R y)$, then $v\in Z$. We write $\mathcal{D}(X,E)$ for the collection of all vertices of $E$ which are determined over $X$. The set $\mathcal{D}(X,E)$ can be constructed inductively by setting $\mathcal{D}(X,E)=X_n$, where: $X_0:=X$; $X_{k+1}$ is the union of $X_k$ together with those vertices $v$ of $E$ for which there are distinct $x,y\in X_k$ with  $E\models (v R x)$ and $E\models (v R y)$; and $n$ is the least natural number with $X_{n+1}=X_n$. 

Let  $A,E,B,C$ be finite graphs with $A\leq E$ and $E\leq B,C$. By definition, in any graph $D$ which is an amalgam 
of $B$ and $C$ over $A$, the vertices of the copy of $A$ in $B$ have to get identified to the associated vertices of the copy of $A$ in $C$.  If
 $D\in \mathcal{K}(C_4)$, then these identifications extend to vertices of $\mathcal{D}( \mathrm{dom}(A),E)$ as well:

\begin{lemma}\label{C:3}
Let  $f,g\colon E\to D$ be two graph embeddings with $D\in \mathcal{K}(C_4)$. Assume that  $f(v)=g(v)$ for all $v\in X\subseteq \mathrm{dom}(E)$. Then,  $f(v)=g(v)$ for every  $v\in \mathcal{D}( X,E)$.
\end{lemma}
\begin{proof}
Assume that $f(x)=g(x)$ and  $f(y)=g(y)$ for some $x,y\in\mathrm{dom}(E)$. If $E\models (v R x)$ and $E\models (v R y)$, then $f(v)=g(v)$, since otherwise we would have the $4$-cycle $f(x),f(v),f(y),g(v)$ in $E$.
\end{proof}

We can now turn to the proof of Theorem \ref{T:2}.

\begin{proof}[Proof of Theorem \ref{T:2}]
Let $A$ be any graph in $\mathcal{K}(C_4)$ with the property that
$A$ does not embed to any $E\in \mathcal{K}(C_4)$ so that $E$ has a dominating vertex. For example, we can take $A:=C_5\in \mathcal{K}(C_4)$ to be the pentagon. We will show that for every 
embedding $i\colon A\to \widehat{A}$,  with $\widehat{A}\in \mathcal{K}(C_4)$, there are embeddings $f\colon \widehat{A}\to B$ and  $g\colon \widehat{A}\to C$, with $B,C\in\mathcal{K}(C_4)$, which do not amalgamate over $A$. This will show  that $\mathcal{K}(C_4)$ fails WAP and hence $\mathcal{C}(C_4)$ does not admit a generic element; see  Fact \ref{F:1}.

Fix any embedding $i\colon A\to \widehat{A}$  with $\widehat{A}\in \mathcal{K}(C_4)$. We view $A$ as a  subgraph of $\widehat{A}$ after renaming some of the vertices, if necessary.

\begin{claim}\label{C:4}
There is   $E\in \mathcal{K}(C_4)$ with $\widehat{A}\leq E$, and  vertices $x,x',y',y\in  \mathcal{D}(\mathrm{dom}(A),E)$ which 
form a geodesic path  of length $3$ in $E$.
\end{claim}
\begin{proof}[Proof of Claim \ref{C:4}]
Let $D$ be the subgraph of $\widehat{A}$ induced on the set
$ \mathcal{D}(\mathrm{dom}(A),\widehat{A})$ and
let  $\{\{x_i,y_i\}\colon i\in I\}$ be the, potentially empty, list of all edges $\{x,y\}$  of $D$ for which there is no $v\in \mathrm{dom}(D)$ so that $x,y,v$ is a $3$-cycle in $D$. Notice that if such  vertex $v$  existed in $\widehat{A}$ then $v$ would have been determined over $A$ and hence $v$ would already be in $D$. If follows that  for all $i\in I$ the edge $\{x_i,y_i\}$ is not contained in a $3$-cycle even in the ambient graph $\widehat{A}$.

We now extend $\widehat{A}$ by adding for every $i\in I$ a new vertex $v_i$ which we connect by an edge only with $x_i$ and with $y_i$. The resulting graph $E$ is clearly in $ \mathcal{K}(C_4)$ and 
$\mathcal{D}(\mathrm{dom}(A),E)=\mathrm{dom}(D)\cup\{v_i \colon i\in I\}$.  Let $D'$ be the subgraph of $E$ induced on the latter set. All edges of $D'$ are contained in some $3$-cycle of $D'$. Also, $D'$ has no dominating vertex since it contains $A$. By Proposition \ref{prop:Blokhuis} there is a geodesic path $x,x',y',y$ of length $3$ in $D'$. This path remains geodesic in $E$ since
any vertex $v$ in $E$ which would witness a potential  ``shortcut"---that is, any $v$ with  $E\models vR x$ and $E\models vR y$---would  have to be determined over $\{x,y\}\subseteq \mathrm{dom}(D')$ in $E$ and hence it would have to already lie in $D'$. 
\end{proof}

Let now $E\supseteq \widehat{A}$ and let $x,x',y',y\in \mathrm{dom}(E)$  as in Claim \ref{C:4}. In particular, we have that
 $x,x',y',y\in\mathcal{D}(\mathrm{dom}(A),E)$.  We extend $E$ to $E'$ by first adding a new vertex $z$ which is connected only to $x$ and $y$. We then add a vertex $w_y$ which is connected only to $x$ and $z$ and a vertex $w_x$ which is connected only to $y$ and $z$; see Figure \ref{F:Final}. We can also assume without loss of generality that there are $v_x,v_y\in \mathrm{dom}(E)$ which are determined over $A$, so that $xRv_x R x'$ and $yRv_y R y'$. Indeed
 the construction in Claim \ref{C:4} guarantees that such $v_x,v_y$ exist. Notice that   $v_x,v_y,w_x,w_y\in \mathcal{D}(\mathrm{dom}(A),E')$.

 We can now define two further extensions $B,C$ of $E'$ (and hence of $\widehat{A}$) which do not amalgamate over $A$. We define $B$  by adding to $E'$ a new vertex $s$ that is connected only with $w_x$ and $v_x$ and a new vertex $t$ that is connected only with $w_y$ and $v_y$. Finally, we also connect $s$ and $t$ by an edge. We have that $B\in\mathcal{K}(C_4)$. This follows from the fact that $E'\models \neg (v_x R v_y)$  which holds since otherwise we would already have a $4$-cycle $x',y',v_y,v_x$  in $E$.
The  definition of the extension $C$  of $E'$  is similar to $B$ with the only difference that instead of connecting $s$ and $t$ with an edge, we connect $s,t$ via a path $s,q,r,t$  of length $3$, where $q,r$ are entirely new vertices.

\begin{figure}[ht!]
\hspace*{-9em}{
\begin{tikzpicture}[x=0.68pt,y=0.75pt,yscale=-1,xscale=1]
\path (0,300); 
\draw   (203,200) node (1){$\bullet$} -- (245,60) node (2){$\bullet$} -- (358,60) node (3){$\bullet$} -- (400,200) node (4){$\bullet$} -- cycle ;
\draw    (224,131) node (5){$\bullet$} -- (258,199) node (6){$\bullet$} ;
\draw    (379,131) node (7){$\bullet$} -- (343,200) node (8){$\bullet$}; 
\draw    (304,60) node (9){$\bullet$}  -- (224,131);
\draw    (304,60) -- (379,131) ;
\draw    (245,60)  .. controls (351,27) and (385,29) .. (450,60) node (10){$\bullet$} ;
\draw    (450,60) .. controls (449,114) and (414,176) .. (400,200) ;
\draw    (152,61) .. controls (192,31) and (298,30) .. (358,60) ;
\draw    (152,61) node (11){$\bullet$}  .. controls (151,120) and (182,172) .. (200,200)  ;

\draw[dashed]    (152,61) .. controls  (225,0)  and (375,0)   .. (450,60) ;

\draw[dotted]    (152,61) .. controls  (225,-60)   and (375,-60)   .. (450,60) ;

\draw  (225,-10)  node (20){$\bullet$} ;
\draw  (225,-22)  node (20'){q} ;

\draw  (375,-10)  node (21){$\bullet$} ;
\draw  (375,-22)  node (21'){r} ;


\draw (298,65) node [anchor=north west][inner sep=0.75pt]    {$z$};
\draw (385,122) node [anchor=north west][inner sep=0.75pt]    {$y$};
\draw (208,122) node [anchor=north west][inner sep=0.75pt]    {$x$};
\draw (260,202) node [anchor=north west][inner sep=0.75pt]    {$x'$};
\draw (345,203) node [anchor=north west][inner sep=0.75pt]    {$y'$};
\draw (200,205) node [anchor=north west][inner sep=0.75pt]    {$v_{x}$};
\draw (394,205) node [anchor=north west][inner sep=0.75pt]    {$v_{y}$};
\draw (361,53) node [anchor=north west][inner sep=0.75pt]    {$w_{x}$};
\draw (220,53) node [anchor=north west][inner sep=0.75pt]    {$w_{y}$};
\draw (140,47) node [anchor=north west][inner sep=0.75pt]    {$s$};
\draw (454,44) node [anchor=north west][inner sep=0.75pt]    {$t$};
\end{tikzpicture}
}
\vspace{-55pt} \label{F:Final}
\caption{The  edge $\{s,t\}$ belongs only to the graph $B$. The vertices $q,r$  and the edges $\{s,q\}, \{q,r\}, \{r,t\}$ belong only to the graph $C$.}
\end{figure}

We claim that $B,C$ do not amalgamate over $A$ to graph in $\mathcal{K}(C_4)$. Indeed, let $D\in \mathcal{K}(C_4)$ and let  $f\colon B\to D$ and $g\colon C\to D$ be embeddings with $f\res A= g\res A$. But then we have that $f(s)=g(s)$ and $f(t)=g(t)$ as well. Indeed,  by Lemma \ref{C:3} and  since $v_x,v_y,w_x,w_y\in \mathcal{D}(\mathrm{dom}(A),E')$, we have that $f(v_x)=g(v_x)$, $f(v_y)=g(v_y)$, $f(w_x)=g(w_x)$, $f(w_y)=g(w_y)$. Hence, if $f(s)\neq g(s)$ then $\{f(s),w_x,g(s),v_x\}$ would be a $4$-cycle in $D$, and  if $f(t)\neq g(t)$ then $f(t),w_y,g(t),v_y$ would be a $4$-cycle in $D$, contradicting that $D\in \mathcal{K}(C_4)$.  But then, since $f(s)=g(s)$ and $f(t)=g(t)$, the edge $\{f(s),f(t)\}$ in $D$ together with the path $g(s),g(q), g(r), g(t)$ in $D$ form a $4$-cycle in  $D$, contradicting that $D\in \mathcal{K}(C_4)$.

\end{proof}

Cherlin and Komj\'{a}th  \cite{CK} extended the results from \cite{HP} by showing that there is no universal element in the class $\mathcal{C}(C_k)$ of all  countable $C_k$-free graphs, for all $k>3$. 

\begin{question}
Does $\mathcal{C}(C_k)$ admit a generic element for some $k>4$? 
\end{question}


\begin{thebibliography}{99}
\bibitem{BB} A. Blokhuis, A. Brouwer, {\em Geodetic graphs of diameter two},  Geom. Dedicata, {\bf 25}, (1988), 527--533. 

\bibitem{Ca} P. Cameron, {\em Combinatorics 3:
Finite geometry and strongly
regular graphs}, \\ https://cameroncounts.files.wordpress.com/2018/10/acnotes3.pdf



\bibitem{CK} G. Cherlin, P. Komj\'{a}th {\em There is no universal countable pentagon free graph},  J. Graph Theory,  {\bf 18}, (1994), 337--342. 


\bibitem{CSS} G. Cherlin, S. Shelah, and N. Shi, {\em Universal graphs with forbidden subgraphs and algebraic closure},  Advances in Applied Mathematics, {\bf 22}, (1999), 454--491.



\bibitem{CT} G. Cherlin, L. Tallgren {\em Universal graphs with a forbidden near-path or 2-bouquet},  J. Graph Theory, {\bf 56}, (2007), no. 2, 41--63.



\bibitem{DF} J. Deutsch, P.H. Fisher {\em On strongly regular graphs with $\mu=1$},  European J. Combin. {\bf 22}, (2001), no. 3, 303--306. 

\bibitem{DHV} R. Diestel, R. Hahn, and  W. Vogler, {\em  Some remarks on universal graphs},  Combinatorica {\bf 5}, (1985),  283--293.



\bibitem{HP} A. Hajnal, J. Pach {\em Monochromatic paths in infinite graphs},  Finite and infinite sets, Proc. Colloq. Soc. J. Bolyai,  North-Holland, Amsterdam, (1981), no. 37, 359--369. 








\bibitem{HN} J. Hubi\v{c}ka, J. Ne\v{s}et\v{r}il {\em Bowtie-free graphs have a Ramsey lift},  Advances in Applied Mathematics {\bf 96}, (2018),  286--311. 




\bibitem{Ivanov}
 A. Ivanov, {\it Generic expansions of $\omega$-categorical structures and semantics of generalized quantifiers}, J. Symbolic Logic {\bf 64}, (1999), 775--789.

\bibitem{KR} A.S. Kechris, C. Rosendal, {\em Turbulence, amalgamation, and generic automorphisms of homogeneous structures},  Proc. Lond. Math. Soc.,  {\bf 94}, (2007), 302--350. 





\bibitem{KMP} P. Komj\'{a}th,  A.H. Mekler, and J. Pach,  {\em Some universal graphs}, Israel J. Math., {\bf 64}, (1988), 158--168. 



\bibitem{KP}  P. Komj\'{a}th,  J. Pach, {\em Universal elements and the complexity of certain classes of infinite graphs}, Discrete Mathematics,  {\bf 95} (1–3), (1991),  255--270. 


\bibitem{Ko} P. Komj\'{a}th,  {\em Some remarks on universal graphs}, Discrete Mathematics, {\bf 199} (1-3),  (1999) 259--265. 


\bibitem{KKgames}
A. Krawczyk, W. Kubi\'s,
{\it Games on finitely generated structures}, Annals of Pure and Applied Logic, {\bf 172}, no. 10, (2021), 13 pp.


\bibitem{WAP}
A. Krawczyk, A. Kruckman, W. Kubi\'s{}, A. Panagiotopoulos, {\it Examples of weak amalgamation classes},  Mathematical Logic Quarterly (to appear).




\bibitem{K} A. Kruckman, {\em Infinitary limits of finite structures}, PhD Thesis, University of California, Berkeley, (2016).



\bibitem{Lachlan} A.H. Lachlan, R. E. Woodrow, {\em Countable ultrahomogeneous undirected graphs},  Trans.  Amer. Math. Soc., {\bf 262} (1), (1980), 51–-94.



\bibitem{KT} Z. Kabluchko, K. Tent, {\em On weak \Fraisse{} limits},  arXiv:1711.09295.



\bibitem{Kechris} S.A. Kechris, {\em Classical descriptive set theory},  Graduate Texts in Mathematics, Springer, New York (2011).




\bibitem{DiLiberti}
 I. Di Liberti, {\it Weak saturation and weak amalgamation property},  J. Symbolic Logic {\bf 84} (3), (2019),  929--936


 \bibitem{M} D.  Macpherson,  {\em A survey of homogeneous structures},
  Discrete Math. {\bf 311} (2011), no. 15, 1599--1634. 


 \bibitem{S} D.  Siniora,  {\em Bowtie-free graphs and generic automorphisms},
  	arXiv:1705.01347. 



\bibitem{Rado}  R. Rado, {\em Universal graphs and universal functions}, Acta Arith. {\bf 9} (1964) 331-340.

\bibitem{TZ} K. Tent, M. Ziegler, {\em A course in model theory}. Lecture Notes in Logic, Cambridge University Press, Cambridge, 2012.

\end{thebibliography}
\end{document}